\documentclass[12pt]{amsart}
\usepackage{amssymb,latexsym,amsmath,amsthm,amscd}
\usepackage{setspace}
\usepackage{verbatim}
\usepackage[active]{srcltx}
\usepackage[all]{xy} \xyoption{arc}
\usepackage[left=3cm,top=2cm,right=3cm,bottom = 2cm]{geometry}
\usepackage{graphicx}
\usepackage[usenames,dvipsnames]{color}
\usepackage{hyperref}

\theoremstyle{plain}
\newtheorem{thm}{Theorem}[section]
\newtheorem{lem}[thm]{Lemma}
\newtheorem{prop}[thm]{Proposition}
\newtheorem{cor}[thm]{Corollary}

\theoremstyle{definition}
\newtheorem{dfn}[thm]{Definition}
\newtheorem{ex}[thm]{Example}

\theoremstyle{remark}
\newtheorem{rmk}[thm]{Remark}

\newcommand{\cA}{\mathcal{A}}

\newcommand{\cL}{\mathcal{L}}
\newcommand{\cM}{\mathcal{M}}

\newcommand{\cO}{\mathcal{O}}

\newcommand{\cS}{\mathcal{S}}

\newcommand{\cV}{\mathcal{V}}
\newcommand{\cW}{\mathcal{W}}

\newcommand{\frakh}{\mathfrak{h}}

\newcommand{\veps}{\varepsilon}

\DeclareMathOperator{\Tr}{Tr}

\DeclareMathOperator{\GL}{GL}

\DeclareMathOperator{\SL}{SL}
\DeclareMathOperator{\PSL}{PSL}

\newcommand*{\df}{\mathrel{\vcenter{\baselineskip0.5ex \lineskiplimit0pt
                     \hbox{\scriptsize.}\hbox{\scriptsize.}}} =}

\providecommand{\abs}[1]{\left\lvert#1\right\rvert}

\providecommand{\twomat}[4]{\left(\begin{matrix}#1&#2\\#3&#4\end{matrix}\right)}
\providecommand{\stwomat}[4]{\left(\begin{smallmatrix}#1&#2\\#3&#4\end{smallmatrix}\right)}

\newcommand{\QQ}{\mathbf{Q}}
\newcommand{\FF}{\mathbf{F}}
\newcommand{\CC}{\mathbf{C}}

\newcommand{\ZZ}{\mathbf{Z}}
\newcommand{\PP}{\mathbf{P}}

\newcommand{\cMbar}{\overline{\cM}}

\newcommand{\cVbar}{\overline{\cV}}

\newcommand{\cSbar}{\overline{\cS}}

\newcommand{\orb}[2]{#1\backslash\!\!\backslash #2}

\newcommand{\sfrac}[2]{\frac{#1}{#2}}
\newcommand{\cref}[1]{\ref{#1}}

\DeclareMathOperator{\Mp}{Mp}

\begin{document}
\title[Generating weights for Weil Representations]
{Generating weights for the Weil representation attached to an even order cyclic quadratic module}

\author{Luca Candelori}
\address{Department of Mathematics, Louisiana State University, Baton Rouge, LA, USA}
\email{lcandelori@lsu.edu}

\author{Cameron Franc}
\address{Department of Mathematics \& Statistics, University of Saskatchewan, \mbox{Saskatoon}, SK, Canada}
\email{franc@math.usask.ca}

\author{Gene S. Kopp}
\address{Department of Mathematics, University of Michigan, Ann Arbor, MI, USA}
\email{gkopp@umich.edu}
\thanks{This work was supported in part by NSF grant DMS-1401224.}

\date{}

\begin{abstract}
We develop geometric methods to study the generating weights of free modules of vector-valued modular forms of half-integral weight, taking values in a complex representation of the metaplectic group. We then compute the generating weights for modular forms taking values in the Weil representation attached to cyclic quadratic modules of order $2p^r$, where $p\geq 5$ is a prime. We also show that the generating weights approach a simple limiting distribution as $p$ grows, or as $r$ grows and $p$ remains fixed.
\end{abstract}



\maketitle

\section{Introduction}

The aim of this paper is to compute the generating weights of modules of modular forms associated to Weil representations attached to finite quadratic modules of order $2p^r$. Our approach is to extend the main geometric results developed in \cite{CandeloriFranc} so that they can be applied in the half-integral weight setting of the present paper. In the half-integral weight setting, one can replace the weighted projective line $\PP(4,6)$ of \cite{CandeloriFranc} with $\PP(8,12)$, and thereby recover the results of \cite{CandeloriFranc} with natural modifications. The modifications are explained in Section \ref{section:vvaluedModularForms}.

The computation of generating weights is a fundamental problem in the theory of vector-valued modular forms (see \cite{Marks}, \cite{CandeloriFranc}, \cite{FrancMason:Structure}), which is equivalent to determining how certain vector bundles decompose into line bundles \cite{CandeloriFranc}. For unitary representations of $\SL_2(\ZZ)$, the only obstacle to such problems is presented by the forms of weight one, while for nonunitary representations, there could be a greater number of problematic weights (although always finitely many by \cite{Mason1} and \cite{CandeloriFranc}). In the half-integral weight setting, the authors had believed that---even if one only considers unitary representations of $\Mp_2(\ZZ)$---the multiplicities of the three critical weights $1/2$, $1$ and $3/2$ would all be difficult to compute. It turns out that one can use results of Skoruppa \cite{Skoruppa}, \cite{SkoruppaCritical} and Serre-Stark \cite{SerreStark}, along with Serre duality, to handle weights $1/2$ and $3/2$ in many cases, and eliminate weight one by imposing parity restrictions (analogous to restricting to representations of $\PSL_2(\ZZ)$ versus odd representations of $\SL_2(\ZZ)$). The general class of Weil representations that we consider in this paper is introduced in Section \ref{s:quadraticmodules} below. Section \ref{s:specialcase} specializes to certain cyclic Weil representations where it is possible to compute the generating weights of the corresponding module of vector-valued modular forms using the results discussed above. The result of these computations can be found in Table \ref{table:weightMultiplicities} of Section \ref{s:specialcase} below.

The most interesting part of these computations turns out to be the evaluation of $\Tr(L)$, where $L$ denotes a so-called \emph{exponent matrix} for a representation $\rho$ of $\Mp_2(\ZZ)$. Such a matrix $L$ satisfies $\rho(T) = e^{2\pi i L}$, where $T$ denotes a lift of $\stwomat 1101$ to $\Mp_2(\ZZ)$. In Section \ref{s:specialcase}, we consider the finite quadratic module $\ZZ/2p^r\ZZ$ with the quadratic form $q(x) = \frac{1}{4p^r}x^2$, where $p \geq 5$ is an odd prime.
 If $\rho$ is the corresponding Weil representation, then one finds that
\[
\Tr(L) = \sum_{x=1}^{2p^r}\left\{-\frac{x^2}{4p}\right\},
\]
where $\{x\} \in [0,1)$ denotes the fractional part of a real number $x$. Theorem \ref{t:TrL} below shows that $\Tr(L)$ is asymptotic to $p^r = \frac{1}{2}\dim \rho$, with a correction term of order $O\left(p^{r/2+\varepsilon}\right)$ given by a simple function of the class number of $\QQ(\sqrt{-p})$, where the precise form of the correction term depends on the congruence class of $p$ mod $8$. The correction term arises from biases among quadratic residues in the interval from $0$ to $p/4$\footnote{Such biases and their connection to class numbers for the interval from $0$ to $p/2$ are well-known, but the corresponding biases for the interval from $0$ to $p/4$ appear to be less well-known (see e.g. \cite{BerndtChowla}).}. 

In Corollary \ref{c:limitingroots}, we observe that as $\dim \rho$ grows, the exponents become uniformly distributed, and the generating weights approach a simple limiting distribution centered around weights $11/2$ and $13/2$. This limiting distribution result is derived from the generating weight formula that is obtained by computing the Hilbert-Poincare series of the module of vector-valued modular forms associated to this Weil representation. It is interesting to ask whether there might be a more conceptual explanation for such asymptotic behaviour of the generating weights. We leave this as an open question.

\emph{Notation and terminology}: Let $e(x)\df e^{2\pi i x}$. Throughout, we also let $\zeta_n \df e^{2\pi i/n}$. If $x \in \QQ$, then let $\{x\}$ denote the unique rational number in $[0,1)$ such that $\{x\} = x$ in $\QQ/\ZZ$. In this paper, all representations are complex and finite-dimensional.

\section{Vector-Valued Modular Forms on the Metaplectic Group}
\label{section:vvaluedModularForms}

We begin by explaining how the results of \cite{CandeloriFranc}, concerning vector-valued modular forms of integral weight, extend to the half-integral weight case with only minor modifications. The metaplectic group $\Mp_2(\ZZ)$ is the unique nontrivial central extension
$$
1\rightarrow \mu_2 \rightarrow \Mp_2(\ZZ) \rightarrow \SL_2(\ZZ) \rightarrow 1.
$$
The elements of $\Mp_2(\ZZ)$ can be written as $\left(A,\phi(\tau)\right)$, where $A = \stwomat abcd \in \SL_2(\ZZ)$, and $\phi$ is a holomorphic function of $\tau \in \frakh:=\{\tau \in \CC: \Im[\tau]>0\}$ with $\phi^2 = c\tau + d$. Multiplication is defined as 
$$
(A_1,\phi_1(\tau))\cdot (A_2,\phi_2(\tau)) = (A_1A_2, \phi_1(A_2\tau)\phi_2(\tau)).
$$
The group $\Mp_2(\ZZ)$ is generated by
$$
T \df \left(\twomat 1101, 1\right), \quad S \df \left(\twomat 0{-1}10, \sqrt{\tau}\right),
$$
where $\sqrt{\tau}$ denotes the principal branch of the square root, possessing positive real part. Note that $T$ is of infinite order, while $S$ is of order 8. Let also
$$
Z \df S^2 = \left(\twomat {-1}00{-1}, i\right), \quad R \df ST = \left(\twomat 0{-1}11, \sqrt{\tau + 1}\right),
$$
which are of of orders 4 and 12, respectively. With this notation the two generating relations of $\Mp_2(\ZZ)$ are $S^2 = Z = R^3$. The metaplectic group $\Mp_2(\ZZ)$ acts on $\frakh$ by the fractional linear transformation action of its quotient $\SL_2(\ZZ)$. Let 
$$
\cM_{1/2} \df \orb{\Mp_2(\ZZ)}{\frakh}
$$
be the quotient in the category of orbifolds (see e.g. \cite{Hain} for definitions), which we call the {\em metaplectic orbifold}. The double covering map $\Mp_2(\ZZ) \rightarrow \SL_2(\ZZ)$ induces an orbifold map $\cM_{1/2} \rightarrow \cM_1 \df \orb{\SL_2(\ZZ)}{\frakh}$ to the modular orbifold. 

\begin{rmk}
As is well-known, the modular orbifold $\cM_1$ is the moduli space of elliptic curves $E/S$ defined over an analytic space. As explained in the first author's Ph.D. thesis \cite{CandeloriThesis}, the metaplectic orbifold $\cM_{1/2}$  can similarly be identified with the moduli space of pairs $(E,Q)$, where $Q$ is a quadratic form on the Hodge bundle of $E/S$. We will not make use of this moduli interpretation in this paper. 
\end{rmk}

The metaplectic orbifold admits a canonical compactification
$$
\cMbar_{1/2}\df \overline{\orb{\Mp_2(\ZZ)}{\frakh}},
$$
obtained by adding the cusp $\infty$, compatible with the canonical compactification $\cMbar_1$ of $\cM_1$. Using the well-known identification of $\cMbar_1$ with the weighted projective line $\PP(4,6)$, it is not hard to show that there is an orbifold isomorphism 
\begin{equation}
\label{eqn:isomorphismWithP(8,12)}
\cMbar_{1/2} \cong \PP(8,12).
\end{equation}
Therefore, $\cMbar_{1/2}$ can be viewed as an orbicurve of genus 0 with a marked point $\infty$, of generic stabilizer of order 4, and two points with stabilizers $\langle S \rangle$ and $\langle R \rangle$ of order $8$ and $12$, respectively.

Let $\rho: \Mp_2(\ZZ) \rightarrow \GL(V)$ denote a complex, finite-dimensional representation. 

\begin{dfn}
A {\em weakly holomorphic $\rho$-valued modular form} of weight $k\in \frac{1}{2}\ZZ$ is a holomorphic function $f: \frakh \rightarrow V$ such that
$$
f(\gamma \tau) = \phi^{2k}\rho(M)\,f(\tau)
$$
for all $M = (\gamma,\phi) \in \Mp_2(\ZZ)$. 
\end{dfn}

Let $\cL_{1/2}$ be the line bundle over $\cM_{1/2}$ defined by the $\cO^{\times}_{\frakh}$-valued 1-cocycle 
$$
(\gamma,\phi) \in \Mp_2(\ZZ) \longmapsto \phi,
$$
and, for any $k\in \frac{1}{2}\ZZ$, let $\cL_{k}\df \left(\cL_{1/2}\right)^{\otimes 2k}$. Let $\cV(\rho)$ be the local system on $\cM_{1/2}$ corresponding to $\rho$. Then, by definition, a $\rho$-valued modular form of half-integral weight $k$ is a holomorphic section of the vector bundle 
$$
\cV_k(\rho)\df \cV(\rho)\otimes_{\cO_{\cM_{1/2}}} \cL_k.
$$
As follows from the same arguments as in \cite{CandeloriFranc}, 
the vector bundles $\cV_k(\rho)$ admit extensions $\cVbar_{k,L}(\rho)$ to $\cMbar_{1/2}$, entirely determined by a linear map $L \in \mathrm{End}(V)$ such that
$$
\rho(T) = e^{2\pi i L}.
$$
We call $L$ a {\em choice of exponents}. If we further require that all the eigenvalues of $L$ have real part in $[0,1)$ (the {\em standard} choice of exponents), we obtain the so-called {\em canonical} extension $\cVbar_{k}(\rho)$. Another typical choice of $L$ is to require all the eigenvalues of $L$ to  have real part in $(0,1]$ (the {\em cuspidal} choice of exponents), in which case we  obtain the cuspidal vector bundle $\cSbar_{k}(\rho)$.

\begin{dfn}
The holomorphic global sections of $\cVbar_{k}(\rho)$ are called {\em holomorphic} $\rho$-valued modular forms of weight $k$. The vector space of all such modular forms is denoted by $M_k(\rho)$. The holomorphic global sections of $\cSbar_{k}(\rho)$ are called ($\rho$-valued) {\em cusp forms} of weight $k$. The vector space of all cusp forms is denoted by $S_k(\rho)$. More generally we write $M_{k,L}(\rho)$ for the space of global sections of $\cVbar_{k,L}(\rho)$, for any choice of exponents $L$.
\end{dfn}

By standard results in algebraic geometry, the spaces $M_{k,L}(\rho)$ are finite-dimensional, for any finite-dimensional complex representation $\rho$. Denote by
$$
M_L(\rho)\df \bigoplus_{k\in \frac{1}{2}\ZZ} M_{k,L}(\rho)
$$ 
the graded module of all $\rho$-valued holomorphic modular forms, viewed as a module over $M(1) = \CC[E_4,E_6]$. By the same arguments as in \cite{CandeloriFranc}, we have the following basic structure theorem concerning $M_L(\rho)$:

\begin{thm}
\label{thm:FreeModuleThm}
Let $\rho: \Mp_2(\ZZ) \rightarrow \GL(V)$ be any complex representation of dimension $n$. Then:
\begin{itemize}
\item[(i)] $M_L(\rho)$ is a free module of rank $n$ over $M(1)$. 
\item[(ii)] If $k_1\leq \ldots \leq k_n$, $k_j \in \frac{1}{2}\ZZ$, are the weights of the free generators, then 
$$
\sum_j k_j = 12\Tr(L).
$$
\item[(iii)] If $\rho$ is unitarizable, and $L$ is a standard choice of exponents, then $0\leq k_j \leq 23/2$. 
\end{itemize}

\end{thm}

\begin{dfn}
\label{dfn:generatingWeights}
Given $\rho$ as in Theorem \ref{thm:FreeModuleThm}, we call the $k_j$ the {\em generating weights} of $M_L(\rho)$. When $L$ is the standard choice of exponents, we also say that the $k_j$ are the {\em generating weights for $\rho$}. 
\end{dfn}

The problem of computing the generating weights is equivalent to that of computing $\dim M_{k,L}(\rho)$ for all $k$, via the identity of formal power series 
$$
\sum_{k\in \frac{1}{2}\ZZ}\dim M_{k,L}(\rho) t^k = \frac{t^{k_1}+ \ldots + t^{k_{n}}}{(1 - t^4)(1-t^6)} \in \ZZ[\![ t^{1/2} ]\!].
$$
A common way to compute $\dim M_{k,L}(\rho)$ is to first compute the Euler characteristic 
$$
\chi( \cMbar_{1/2},\cVbar_{k,L}(\rho))\df \dim H^0(\cMbar_{1/2},\cVbar_{k,L}(\rho))- \dim H^1(\cMbar_{1/2},\cVbar_{k,L}(\rho))
$$
using the Riemann-Roch Theorem, and then study the problem of vanishing of the $H^1$-term using a variety of techniques. Using the identification $\eqref{eqn:isomorphismWithP(8,12)}$ of $\cMbar_{1/2}$ with $\PP(8,12)$, the Riemann-Roch Theorem for Deligne-Mumford stacks (e.g. \cite{Edidin}) gives the following formula: 

\begin{prop}
\label{prop:generalEulerCharacteristicFormula}
Let $\rho: \Mp_2(\ZZ) \rightarrow \GL(V)$ be a finite dimensional  complex representation, and let $L\in \mathrm{End}(V)$ be a choice of exponents. For $i=\sqrt{-1}$ and $j=0,1,2,3$, let $L^{i^j}$ be the restriction of $L$ to the eigenspace with eigenvalue $i^j$ of the linear map $\rho(Z)$. Then, for all $k\in \frac{1}{2}\ZZ$, 
\begin{equation*}
\begin{aligned}
&\chi( \cMbar_{1/2},\cVbar_{k,L}(\rho)) = \dim(\rho)\frac{(5+k)}{48} - \frac{\Tr(L)}{4} + \\
&+ (-1)^{2k}\left(\frac{5+k}{48}\Tr(\rho(I,-1))  - \frac{1}{4}\Tr(L^{+1}) + \frac{1}{4}\Tr(L^{-1})\right) + \\
& + i^{2k}\left(\frac{5+k}{48}\Tr(\rho(Z)) - \frac{1}{4}\sum_{j=0}^3 i^j \Tr(L^{i^j})\right) + i^{-2k}\left(\frac{5+k}{48}\Tr(\rho(Z^{-1})) - \frac{1}{4}\sum_{j=0}^3 i^{-j} \Tr(L^{i^{-j}})\right)+ \\
&+ \sum_{j=1,3,5,7} \frac{\sqrt{i}^{2k j}}{16}\Tr(\rho(S)^j) + \sum_{j=1,2,4,5,7,8,10,11} \frac{\zeta_{12}^{2kj}}{12(1-\zeta_3^j)}\Tr(\rho(R)^j).
\end{aligned}
\end{equation*}

\end{prop}
\begin{proof}
This is simply an application of Riemann-Roch in our setting. See \cite{CandeloriFranc} and \cite{Edidin} for more details on these sorts of computations.
\end{proof}

\begin{rmk}
Similar formulas for $\chi( \cMbar_{1/2},\cVbar_{k,L}(\rho))$, with $L$ a standard (or cuspidal) choice of exponents have already been obtained (e.g. \cite{Borcherds:Lorentzian}, \S 7 or \cite{SkoruppaCritical}, Theorem 6) by using Selberg's trace formula. Our approach using Riemann-Roch for weighted projective lines seems new. One advantage of our formula is that it is expressed entirely in terms of traces, and thus it can be easily computed from character tables. 
\end{rmk}

In general, it can be shown using Kodaira-type vanishing theorems and Serre duality that the $H^1$-term is zero for all but finitely many $k$ (see \cite{FrancMason:Structure} for the integral weight case, but the same arguments apply in our case. Note also that the arguments of \cite{FrancMason:Structure} yield effective bounds on the generating weights). For the remaining \emph{critical weights}, for which $H^1$ does not necessarily vanish, there is no general method to compute $\dim M_{k,L}(\rho)$, and thus there is no general method to compute the generating weights of $M_L(\rho)$. For example, if $\rho$ is unitarizable and $L$ is the standard choice of exponents, the critical weights are $k=1/2,1,3/2$. The weights 1/2 and 3/2 require special arguments, as we will illustrate below (following \cite{SkoruppaCritical}) for a special class of unitary representations. However, no general method to compute $\dim M_1(\rho)$ is known. For general possibly nonunitary $\rho$ and $L$, the number of critical weights increases linearly with $\dim \rho$ (see \cite{FrancMason:Structure} for integral weights), and  there are no general methods to compute $\dim M_k(\rho)$ in these expanded critical ranges.

\section{Finite Quadratic Modules and the Weil representation}
\label{s:quadraticmodules}

We now recall how to construct representations of $\Mp_2(\ZZ)$ from quadratic forms, and discuss the problem of finding generating weights for the corresponding modules of vector-valued modular forms. 

\begin{dfn}
A {\em finite quadratic module} is a pair $(D,q)$ consisting of a finite abelian group $D$ together with a quadratic form $q:D\rightarrow \QQ/\ZZ$, whose associated bilinear form we denote by
$
b(x,y) \df q(x+y) - q(x) - q(y).
$
The {\em level} of $(D,q)$ is the smallest integer $N$ such that $Nq(x) \in \ZZ$ for all $x\in D$.   We will sometimes abuse notation and use $D$ to stand for $(D,q)$ when there is no chance of confusion.
\end{dfn}

There is an natural notion of orthogonal sum $(D_1,q_1) \oplus (D_2,q_2)$ of quadratic modules. In fact, any finite quadratic module can be decomposed into a direct sum of the following indecomposable components (see \cite{Stromberg}, \S 2)
\begin{align*}
A_{p^k}^t &\df \left( \ZZ/p^k\ZZ,\; \frac{t\,x^2}{p^k} \right), \quad p>2 \text{ prime}, \quad (t,p)=1, \\
A_{2^k}^t &\df \left( \ZZ/2^k\ZZ,\; \frac{t\,x^2}{2^{k+1}} \right), \quad \quad (t,2)=1, \\
B_{2^k} &\df \left( \ZZ/2^k\ZZ\oplus \ZZ/2^k\ZZ ,\frac{x^2 + 2xy + y^2}{2^{k}} \right), \\
C_{2^k} &\df \left( \ZZ/2^k\ZZ\oplus \ZZ/2^k\ZZ ,\frac{xy}{2^{k}} \right). \\
\end{align*}
In the following, we will only be concerned with the first two types of finite quadratic modules, and we will omit the superscript $t$ when $t=1$.

For any $a\in \ZZ$, define 
$$
\Omega_D(a)\df \frac{1}{\sqrt{|D|}} \sum_{x\in D} e(a\,q(x)).
$$
By the standard theory of quadratic Gauss sums, $\Omega_D(1)$ is an eighth root of unity. Define the {\em signature} of $D$ to be the integer mod 8 defined by
$$
\Omega_D(1) = \sqrt{i}^{\rm{sig(D)}}.
$$

For $(D,q)$ a finite quadratic module, let $\CC[D]$ be the $\CC$-vector space of functions $f:D\rightarrow \CC$. This space has a canonical basis $\{\delta_x\}_{x\in D}$ of delta functions, where $\delta_{x}(y) = \delta_{xy}$ (the Kronecker delta). The {\em Weil Representation} 
$$
\rho_D: \Mp_2(\ZZ) \longrightarrow \GL(\CC[D])
$$
is defined with respect to this basis by 
\begin{align*}
\rho_D(T)(\delta_{x}) &= e(-q(x))\delta_{x}, \\
\rho_D(S)(\delta_{x}) &= \frac{\sqrt{i}^{\rm{sig(D)}}}{\sqrt{|D|}}\sum_{y\in D} e(b(x,y))\delta_{y}.
\end{align*}

For any half-integer $k\in \frac{1}{2}\ZZ$, consider the vector bundle
$$
\cW_k(D)\df \cVbar_k(\rho_D)$$
over the compactified metaplectic orbifold $\cMbar_{1/2}$. As explained in Section \ref{section:vvaluedModularForms}, the holomorphic global sections of this vector bundle are vector-valued modular forms taking values in the Weil representation, which have been studied extensively in the context of lattices and theta lifting (e.g. \cite{Borcherds},  \cite{Borcherds:Lorentzian}). Moreover, when $D = A_m = (\ZZ/m\ZZ, x^2/2m)$ for $m\in 2\ZZ_{>0}$, the spaces $M_k(\rho_D)$ are canonically isomorphic to those of Jacobi forms of weight $k + 1/2$ and index $m/2$ (\cite{EichlerZagier}). 

\begin{rmk}
Conventions vary on whether $\rho_D$ or $\rho_D^*$ should be called {\em the} Weil representation attached to $(D,q)$. Following our conventions, the vector of theta constants of a positive-definite lattice belongs to $M_{1/2}(\rho_D^*)$. Authors such as Borcherds (e.g. \cite{Borcherds}, \cite{Borcherds:Lorentzian}) define our $\rho_D^*$ to be {\em the} Weil representation. 
\end{rmk}

Note first that, if $f \in M_k(\rho_D)$, then
$$
f(\tau) = (-1)^{2k+\rm{sig}(D)}f(\tau),
$$
as follows from considering the action of $(I,-1)\in \Mp_2(\ZZ)$. Therefore 
\begin{equation}
\label{eqn:parityArgument}
2k+\mathrm{sig}(D) \text{ is odd } \Longrightarrow M_k(\rho_D)=0.
\end{equation}
In other words, there are no modular forms of integral weight when the signature is odd, and there are no modular forms of half-integral weight when the signature is even.

We may now specialize Proposition \ref{prop:generalEulerCharacteristicFormula} to the representations $\rho_D$, in which case the formula simplifies significantly.

\begin{prop}
\label{prop:WeilRepEulerCharacteristic}
If $\mathrm{sig}(D) + 2k$ is even, then 
\begin{equation*}
\begin{aligned}
&\chi( \cMbar_{1/2},\cW_k(D)) = |D|\frac{(5+k)}{24} - \frac{\Tr(L)}{2}  +\\
& + \Re\left[i^{2k+\mathrm{sig}(D)}\left(\frac{5+k}{24}|D[2]| - \frac{1}{2}\Tr(L|_{\CC[D[2]]})\right)\right]  +\\
&+ \frac{1}{4}\Re\left[\sqrt{i}^{2k+\rm{sig}(D)}\,\Omega_D(2)\right]+ \frac{1}{3}\Re\left[ \frac{\zeta_{12}^{2k}}{(1-\zeta_3)}i^{\rm{sig}(D)}+\frac{\zeta_{6}^{2k}}{(1-\zeta^2_3)}\sqrt{i}^{\rm{sig(D)}}\Omega_D(3) \right],
\end{aligned}
\end{equation*}
where $D[2]$ denotes the $2$-torsion of the finite abelian group $D$.
\end{prop}

\begin{proof}
We apply Proposition \ref{prop:generalEulerCharacteristicFormula} and then simplify. First note that $L = L^{+1}$ if $\rm{sig}(D)$ is even and $L = L^{-1}$ if $\rm{sig}(D)$ is odd. Together with the fact that $2k + \mathrm{sig}(D)$ is even, this justifies the first line of the formula. For the second line, we have $\rho_D(Z) = \rho_D(S)^2 = i^{\rm{sig(D)}}Z_0$, where $Z_0$ is the matrix of order 2 that sends $\delta_{\nu} \mapsto \delta_{-\nu}$. We thus have
$$
\Tr(\rho_D(Z)) = i^{\rm{sig(D)}}|\{ \nu \in D: \nu = -\nu\}| = i^{\rm{sig(D)}}\abs{D[2]}. 
$$
Now $Z_0$ only has eigenvalues of $\pm 1$, with eigenspaces
$$
V^+ = \mathrm{span}\{\delta_{\nu} + \delta_{-\nu}\},\quad V^- = \mathrm{span}\{\delta_{\nu} - \delta_{-\nu}\},
$$
so the difference of the traces of $L$ restricted to the $i^j$ eigenspaces is as stated. For the third line of the equation, note that $\rho_D(S) = \Omega_D(1)\,\,S_0$, where $S_0$ is a matrix of order 4 and $\Omega \df \Omega_D(1)= \sqrt{i}^{\rm{sig(D)}}$ is an $8^{\rm th}$ root of unity. Therefore,
$$
\rho_D(S)^3 = \Omega^3S_0^{-1}, \quad \rho_D(S)^5 = \Omega^5\,S_0, \quad \rho_D(S)^7 = \Omega^7\,S_0^{-1},
$$
and 
\begin{align*}
\sum_{j=1,3,5,7} \sqrt{i}^{2k j}\Tr(\rho_D(S)^j) &= \sqrt{i}^{2k}\Omega\Tr(S_0) + \sqrt{i}^{6k}\Omega^3\Tr(S_0^{-1}) \\ &\ \ \ \ + \sqrt{i}^{10k}\Omega^5\Tr(S_0) + \sqrt{i}^{14k}\Omega^7\Tr(S_0^{-1})\\
&= (\sqrt{i}^{2k}\Omega +\sqrt{i}^{10k}\Omega^5) \Tr(S_0) + \overline{(\sqrt{i}^{2k}\Omega +\sqrt{i}^{10k}\Omega^5) \Tr(S_0)} \\
&= 2\Re[\sqrt{i}^{2k}(\Omega +(-1)^{2k}\Omega^5) \Tr(S_0)].
\end{align*}
If $\rm{sig}(D)$ and $2k$ are both odd, then $\Omega$ is a primitive $8^{\rm th}$ root of unity, so that $\Omega^5 = -\Omega$ and the expression simplifies as 
\begin{align*}
2\Re[\sqrt{i}^{2k}(1 +(-1)^{2k+1}) \Omega\Tr(S_0)] &= 4\Re[\sqrt{i}^{2k}\,\frac{\sqrt{i}^{\rm{sig(D)}}}{\sqrt{D}}\sum_{x\in D} e(2q(x))] \\&= 4\Re\left[\sqrt{i}^{2k+\rm{sig}(D)}\,\Omega_D(2)\right].
\end{align*}
If $\mathrm{sig}(D)$ and $2k$ are both even, then $\Omega^5 = \Omega$, and we get the same expression. 
Suppose now again that $\mathrm{sig}(D)$ and $2k$ are both odd. Then $\rho_D(R) = \zeta_{12}\,\,R_0$, where $R_0$ is a matrix of order 6. Consequently, we have 
\begin{align*}
\sum_{j=1,2,4,5,7,8,10,11} \frac{\zeta_{12}^{2kj}}{12(1-\zeta_3^j)}\Tr(\rho_D(R)^j) = 2\Re[\sum_{j=1,2,4,5} \frac{\zeta_{12}^{2kj}}{12(1-\zeta_3^j)}\Tr(\zeta_{12}^jR_0^j) ],
\end{align*}
where we have used the fact that $\Tr(R_0^{-1}) = \Tr(R_0^{-1,T}) = \overline{\Tr(R_0)}$. Moreover, since $\zeta_{12}^5 = -\zeta_{12}^{-1}$ and $2k$ is odd, we have
$$
\frac{\zeta_{12}^{2k}}{12(1-\zeta_3)}\Tr(\zeta_{12}R_0) + \frac{\zeta_{12}^{10k}}{12(1-\zeta^2_3)}\Tr(\zeta^5_{12}R^5_0) = 2\Re[ \frac{\zeta_{12}^{2k}}{12(1-\zeta_3)}\Tr(\zeta_{12}R_0)],
$$
and 
$$
\frac{\zeta_{6}^{2k}}{12(1-\zeta^2_3)}\Tr(\zeta_{6}R_0) + \frac{\zeta_{6}^{4k}}{12(1-\zeta_3)}\Tr(\zeta^2_{6}R^4_0) = 2\Re[ \frac{\zeta_{6}^{2k}}{12(1-\zeta^2_3)}\Tr(\zeta_{6}R^2_0)]
$$
because $\zeta_6^2 = -\zeta_6^{-1}$. It is clear from the definitions that
$$
\Tr[\rho_D(R)] = \Omega^2 = i^{\rm{sig(D)}},
$$
while a short computation shows that
$$
\Tr[\rho_D(R^2)] = \Tr[\rho_D(ST^{-1})] =  \frac{\sqrt{i}^{\rm{sig(D)}}}{\sqrt{|D|}}\sum_{x\in D} e(3q(x)) = \sqrt{i}^{\rm{sig(D)}}\Omega_D(3).
$$
If $\mathrm{sig}(D)$ and $2k$ are both even, then $\rho_D(R)$ is already of order 6 and the same argument goes through to obtain the same formula. 
\end{proof}

\begin{ex}
Let $p>3$ be a prime and consider the finite quadratic module
$$
D = A_p := (\ZZ/p\ZZ, x^2/p),
$$
which has $|D| = p$. We show how to compute the Euler characteristic of $\cW_k(A_p)$ using Proposition \ref{prop:WeilRepEulerCharacteristic}. By a classical Gauss sum computation,
$$
\Omega_{A_p}(1) = \left\{ \begin{array}{lc}
 1 & \text{ if } p\equiv 1 (4) \\
 i & \text{ if } p\equiv 3 (4)  \\ 
 \end{array} \right.,
$$
so that the signature is even in all cases. Since $(2,p)=1$, we have $A_p[2]= \{0\}$.  Thus, $\Tr(\rho_D(Z)) = i^{\mathrm{sig(D)}}$, and $\CC[A_p[2]] = \delta_0$, so that $\Tr(L|\delta_0) = 0$. Moreover, by standard Gauss sum arguments, 
$$
\Omega_{A_p}(1)\Omega_{A_p}(a) = \Omega_{A_p}^2(1)\left ( \frac{a}{p} \right ) = \left ( \frac{-a}{p} \right ), \quad a=2,3,
$$
which takes care of the terms coming from $\Tr(\rho_{A_p}(S))$ and $\Tr(\rho_{A_p}(R^2))$. It remains to compute the trace of $L$, the standard choice of exponents. To this end, note that 
$$
  \Tr(L) = \sum_{x=1}^{p-1} \left\{-\frac{x^2}{p}\right\}
=2\sum_{\substack{a=1\\ \chi_p(a) = 1}}^{p-1}\frac{p-a}{p} = (p-1) - 2\sum_{\substack{a=1\\ \chi_p(a) = 1}}^{p-1}\frac{a}{p}.
$$
Now by Dirichlet's class number formula, 

\[
\sum_{\substack{a=1\\ \chi_p(a) = 1}}^{p-1} \frac{a}{p}= \begin{cases}
\frac{1}{4}(p-1) & p \equiv 1 \pmod{4},\\
\frac{1}{4}(p-1)- \frac{h_p}{2} & p\equiv 3 \pmod{4},
\end{cases}
\]
with $h_p$ being the class number of $\QQ(\sqrt{-p})$. Therefore 
\[
  \Tr(L)= \begin{cases}
\frac{1}{2}(p-1) & p \equiv 1 \pmod{4},\\
\frac{1}{2}(p-1)+ h_p & p\equiv 3 \pmod{4}.
\end{cases}
\]
Putting everything together, for $p\equiv 1 \pmod{4}$ and $k\in\ZZ$ we have
\begin{equation*}
\begin{aligned}
&\chi( \cMbar_{1/2},\cW_k(A_p)) = \frac{p(5+k)}{24} - \frac{p-1}{4}  +\\
& + \Re\left[i^{2k}\frac{5+k}{24}\right]  
+ \frac{1}{4}\Re\left[\sqrt{i}^{2k}\left(\frac{2}{p}\right)\right]+ 
\frac{1}{3}\Re\left[\frac{\zeta_{12}^{2k}}{(1-\zeta_3)}+\frac{\zeta_{6}^{2k}}{(1-\zeta^2_3)}\left(\frac{3}{p}\right) \right]
\end{aligned}
\end{equation*}
while, for $p\equiv 3 \pmod{4}$ and $k\in\ZZ$, we have
\begin{equation*}
\begin{aligned}
&\chi( \cMbar_{1/2},\cW_k(A_p)) = \frac{p(5+k)}{24} - \frac{p-1}{4} - \frac{h_p}{2}  +\\
& - \Re\left[i^{2k}\frac{5+k}{24}\right]  
- \frac{1}{4}\Re\left[\sqrt{i}^{2k}\left(\frac{2}{p}\right)\right]- 
\frac{1}{3}\Re\left[\frac{\zeta_{12}^{2k}}{(1-\zeta_3)}+\frac{\zeta_{6}^{2k}}{(1-\zeta^2_3)}\left(\frac{3}{p}\right) \right].
\end{aligned}
\end{equation*}
The cases $p=2,3$ can similarly be worked out, by taking special care of the Gauss sums involved and by dividing the class numbers by the appropriate number of roots of unity.  
\end{ex}

Next, we discuss the general problem of computing the generating weights (Definition \ref{dfn:generatingWeights}) for the free module $M(\rho_D)$ of the Weil representation attached to the finite quadratic module $(D,q)$. As explained in Section \ref{section:vvaluedModularForms}, this is equivalent to computing $\dim M_k(\rho_D)$ for all $k\in \frac{1}{2}\ZZ$.

First, since $\rho(D)$ is unitary, we have (Theorem \ref{thm:FreeModuleThm}, (iii))
$$
\dim M_k(\rho_D) = 0,\quad k<0, 
$$
and (e.g. \cite{CandeloriFranc}, \S 6)
$$
\dim M_0(\rho_D) = \dim_{\CC} \mathrm{Inv}(\rho(D)).
$$

Moreover, by Serre duality,
$$
\dim H^1(\cMbar_{1/2},\cW_k(D)) = \dim S_{2-k}(\rho_D^*) \leq \dim M_{2-k}(\rho_D^*).
$$
Thus,
\begin{align*}
\dim M_k(\rho_D) &= \chi(\cMbar_{1/2},\cW_k(D)), \quad k>2,
\end{align*}
since $\rho_D^*$ is also unitary (and thus $M_{2-k}(\rho_D^*) = 0$ if $k>2$). For $k=2$, we have $\dim S_{0}(\rho_D^*) = 0$, and therefore $\dim M_2(\rho_D) = \chi(\cMbar_{1/2},\cW_2(D))$ as well. It follows that $\dim M_k(\rho_D)$ can be readily computed from Proposition \ref{prop:WeilRepEulerCharacteristic} in all weights $k\leq 0$ and $k\geq 2$. It remains then to consider the `critical weights' $k=1/2,1,3/2$. For $k=1/2$, we have the following formula of Skoruppa.

\begin{thm}[\cite{SkoruppaCritical}, Theorem 9]
\label{theorem:CriticalWeights}
Let $D$ be a finite quadratic module of level dividing $4m$. Then 
\begin{align*}
\dim M_{1/2}(\rho_D) &= \dim \bigoplus_{\substack{l\mid m \\ m/l \text{ squarefree}}} \mathrm{Inv}(\rho_{A^{-1}_{2l}\oplus D^{-1}})^{\epsilon\times 1}, \\
\dim S_{1/2}(\rho^*_D)  &= \dim \bigoplus_{\substack{l\mid m \\ m/l \text{ squarefree}}} \mathrm{Inv}(\rho_{A^{-1}_{2l}\oplus D})^{\epsilon\times 1} - \dim \bigoplus_{\substack{l\mid m \\ l/m \text{ squarefree}}} \mathrm{Inv}(\rho_{A^{-1}_{2l}\oplus D})^{\mathrm{O}(A^{-1}_{2l})\times 1}, \\
\end{align*}
where $\mathrm{O}(D)$ is the orthogonal group of the finite quadratic module $(D,q)$, acting naturally on $\CC[D]$, and $\epsilon \in \GL(\CC[D])$ is given by $\delta_{\nu}\mapsto \delta_{-\nu}$. 
\end{thm}

Note that by Serre duality we have
$$
\dim M_{3/2}(\rho_D) = \chi(\cMbar_{1/2},\cW_{3/2}(D)) + \dim S_{1/2}(\rho^*_D),
$$
and thus Theorem \ref{theorem:CriticalWeights} can also be used to compute  $\dim M_{3/2}(\rho_D)$.

Finally, for weight $k=1$, there are no known general formulas for $M_1(\rho_D)$. This is presently the only obstruction to computing the generating weights of all Weil representations. In particular, if $(D,q)$ is of odd signature, then the generating weights can {\em always} be computed, since in this case $M_1(\rho) = 0$ by \eqref{eqn:parityArgument}. In the rest of paper, we analyze the case $D=A_{2p^r}$, where $p>3$ is a prime, but in principle the same methods can be applied to any finite quadratic module of odd signature.

\section{The case $D = A_{2p^r}$}
\label{s:specialcase}

We now compute the generating weights for the Weil representation $\rho_D$, where 
$$
D = A_{2p^r} = \left(\ZZ/2p^r\ZZ, x \mapsto \frac{x^2}{4p^r}\right),\quad p \geq 5\text{ a prime},
$$
and discuss their distribution as $p\to\infty$. First, note that $\Omega_{A_{2p^r}}(1) = \sqrt{i}$, so that $\mathrm{sig}(A_{2p^r}) = 1$ is odd, and we are guaranteed to succeed in computing the generating weights because $M_1(\rho_D)=0$. We then compute $\chi(\cMbar_{1/2}, \cW(A_{2p^r}))$ using Proposition \ref{prop:WeilRepEulerCharacteristic}, and we simplify the terms involving $\Tr(L)$ and the Gauss sums. This will allow us to compute the generating weights efficiently for large $p^r$ where computing the sums in Proposition \ref{prop:WeilRepEulerCharacteristic} na\"ively would be prohibitive. This simplification will also allow us to understand the asymptotic behaviour of the generating weights.

\subsection{Computation of $\Tr(\rho_D(S))$ and $\Tr(\rho_D(R^2))$}
Set $D = A_{2p^r}$ as above. To compute $\Tr(\rho_D(S))$, first separate $D$ into Jordan components,
$$
D \cong A_{2}^{p^r} \oplus A_{p^r}^{4}.
$$
Note that 
$$
\Tr(\rho_{A_{2}^{p^r}}(S)) = ( 1^{p^r} + (-1)^{p^r} )/\sqrt{2} = 0,
$$
since $p$ is odd. But $\rho_D \cong \rho_{A_{2}^{p^r}}\otimes\rho_{A_{p^r}^{4}}$ and therefore $\Tr(\rho_D(S))=0$. 

To compute the trace of $\rho_D(R^2)$, write again
$$
\Tr(\rho_D(R^2)) = \Tr(\rho_{A_{2}^{p^r}}(R^2))\Tr(\rho_{A_{p^r}^4}(R^2))
$$
as above. For each term, we have
\begin{align*}
\Tr(\rho_{A_{p^r}^4}(R^2)) &= \Omega_{A_{p^r}^4}(1)\Omega_{A_{p^r}^4}(3) \\
 &= \left(\frac{4}{p^r}\right)^2\Omega_{A_{p^r}}(1)\Omega_{A_{p^r}}(3) \\
 & = \left(\frac{3}{p^r}\right)\Omega_{A_{p^r}}(1)^2 \\
  & = \left(\frac{-3}{p^r}\right), 
\end{align*}
and
\begin{align*}
\Tr(\rho_{A_{2}^{p^r}}(R^2)) &= \Omega_{A_{2}^{p^r}}(1)\Omega_{A_{2}^{p^r}}(3) \\
&= \left(\frac{2}{p^r}\right)\sqrt{i}^{p^r}\cdot\left(\frac{2}{3p^r}\right)\sqrt{i}^{3p^r} \\
&= - \left(\frac{2}{3}\right) 
\end{align*}
so that
$$
\Tr(\rho_D(R^2)) = - \left(\frac{2}{3}\right)\left(\frac{-3}{p^r}\right) = \left(\frac{p^r}{3}\right),
$$
the last equality coming from quadratic reciprocity. 

\subsection{Computation of $\Tr(L)$}

Next we compute $\Tr(L)$ for $m=2p^r$ for $p$ an odd prime. First we must prove a few lemmas to deal with certain sums involving imprimitive Dirichlet characters.

\begin{lem}
\label{lem:imprim1}
Let $\chi$ be a nontrivial Dirichlet character modulo $m$, and suppose $m \mid n$.  Then
\begin{equation*}
\sum_{a=0}^{n-1} \chi(a) \frac{a}{n} = \sum_{s=0}^{m-1} \chi(s) \frac{s}{m}.
\end{equation*}
\end{lem}
\begin{proof}
Write $n=dr$. Then
\begin{align*}
\sum_{a=0}^{n-1} \chi(a) \frac{a}{n} &=
\sum_{q=0}^{r-1} \sum_{s=0}^{m-1} \chi(s) \frac{mq+s}{mr} \\
&= \left(\sum_{q=0}^{r-1} \frac{q}{r}\right)\left(\sum_{s=0}^{m-1} \chi(s)\right)  + \frac{1}{r}\left(\sum_{q=0}^{r-1} 1\right) \left(\sum_{s=0}^{m-1} \chi(s) \frac{s}{m}\right) \\
&= \left(\sum_{q=0}^{r-1} \frac{q}{r}\right) \cdot 0  + \frac{1}{r} \cdot r \cdot \left(\sum_{s=0}^{m-1} \chi(s) \frac{s}{m}\right) \\
&= \sum_{s=0}^{m-1} \chi(s) \frac{s}{m}.
\end{align*}
\end{proof}

\begin{lem}\label{lem:imprim2}
Let $\chi$ be a nontrivial Dirichlet character modulo $m$, suppose $m \mid n$, and let $\chi^\ast$ be the associated imprimitive character modulo $n$ (which may differ from $\chi$ by vanishing at integers sharing prime factors with $n$ but not $m$).  Then:
\begin{equation*}
\sum_{a=0}^{n-1} \chi^\ast(a) \frac{a}{n} = \left(\sum_{d\mid\frac{n}{m}} \mu(d)\chi(d)\right)\left(\sum_{s=0}^{m-1} \chi(s) \frac{s}{m}\right),
\end{equation*}
where $\mu$ is the M\"{o}bius function.
\end{lem}
\begin{proof}
Follows from Lemma \ref{lem:imprim1} and M\"{o}bius inversion.
\end{proof}

For an odd prime $p$, define
\begin{align*}
h_p^\ast &\df -\sum_{a=0}^{p-1} \left(\frac{a}{p}\right)\frac{a}{p},\\
h_p^\prime &\df -\sum_{\substack{a=0\\ a\textrm{ odd}}}^{4p-1} (-1)^\frac{a-1}{2}\left(\frac{a}{p}\right)\frac{a}{4p}.
\end{align*}
Character sums of this form vanish for odd characters, so in particular, $h_p^\ast=0$ when $p \equiv 1 \mod 4$, and $h_p^\prime=0$ when $p \equiv 3 \mod 4$.  On the other hand, the Dirichlet class number formula determines these values in the remaining cases.  We have
\begin{align*}
h_p^\ast &= \left\{\begin{array}{cl}
0 & \mbox{ if } p \equiv 1 \mod 4 \\
h_p & \mbox{ if } p \equiv 3 \mod 4,~ p \neq 3 \\
\frac{1}{3} & \mbox{ if } p = 3 
\end{array}\right., \\
h_p^\prime &= \left\{\begin{array}{cl}
h_p & \mbox{ if } p \equiv 1 \mod 4 \\
0 & \mbox{ if } p \equiv 3 \mod 4
\end{array}\right..
\end{align*}

Most of the work of computing $\Tr(L)$ is contained in the next two propositions.

\begin{prop}\label{prop:sum1}
\begin{equation}\label{eq:sum1}
\sum_{x=0}^{p^r-1} \left\{\frac{x^2}{p^r}\right\} = \frac{1}{2}\left(p^r-p^{\left\lfloor\frac{r}{2}\right\rfloor}\right) - \frac{p^{\left\lfloor\frac{r+1}{2}\right\rfloor}-1}{p-1}h_p^\ast.
\end{equation}
\end{prop}
\begin{proof}
In general, let $s_n(y)$ denote the number of solutions to $x^2 \equiv y \mod{n}$.  If $p\nmid a$ and $0 < ap^\ell < p^r$, then
\[
s_{p^r}\left(ap^\ell\right) = \begin{cases}
p^{\ell/2} \left(1+\left(\frac{a}{p}\right)\right) & \mbox{ if } 2\mid\ell \\
0 & \mbox{ if } 2\nmid\ell
\end{cases}.
\]
\begin{align}
\sum_{x=0}^{p^r-1} \left\{\frac{x^2}{p^r}\right\} 
&= \sum_{y=0}^{p^r-1} s_{p^r}(y) \frac{y}{p^r} \nonumber \\
&= \sum_{\ell=0}^{r-1} \sum_{\substack{0<a<p^{r-\ell}\\p\nmid a}} s_{p^r}(ap^\ell) \frac{a}{p^{r-\ell}} \nonumber \\
&= \sum_{t=0}^{\lfloor\frac{r-1}{2}\rfloor} \sum_{\substack{0<a<p^{r-2t}\\p\nmid a}} p^t \left(1+\left(\frac{a}{p}\right)\right) \frac{a}{p^{r-2t}} \nonumber \\
&= \sum_{t=0}^{\lfloor\frac{r-1}{2}\rfloor} p^t\left(\sum_{\substack{0<a<p^{r-2t}\\p\nmid a}} \frac{a}{p^{r-2t}}  + \sum_{\substack{0<a<p^{r-2t}\\p\nmid a}} \left(\frac{a}{p}\right)\frac{a}{p^{r-2t}} \right).\label{eq:shortsum}
\end{align}
The first internal sum is 
\begin{align*}
\sum_{\substack{0<a<p^{r-2t}\\p\nmid a}} \frac{a}{p^{r-2t}}
&= \sum_{a=0}^{p^{r-2t}-1} \frac{a}{p^{r-2t}} - \sum_{b=0}^{p^{r-2t-1}-1} \frac{b}{p^{r-2t-1}} \\
&= \frac{1}{2}\left(p^{r-2t}-1\right) - \frac{1}{2}\left(p^{r-2t-1}-1\right) \\
&= \frac{1}{2}\left(p^{r-2t}-p^{r-2t-1}\right).
\end{align*}
The second internal sum is, by Lemma \cref{lem:imprim1},
\[
\sum_{\substack{0<a<p^{r-2t}\\p\nmid a}} \left(\frac{a}{p}\right)\frac{a}{p^{r-2t}}= \sum_{a=0}^{p-1} \left(\frac{a}{p}\right)\frac{a}{p} = h_p^\ast.
\]
After plugging the values of these internal sums to (\cref{eq:shortsum}), we obtain an expression involving a telescoping sum and a geometric series.  This expression simplifies to (\cref{eq:sum1}).
\end{proof}

\begin{prop}\label{prop:sum2}
\begin{align}
\sum_{\substack{0<x<2p^{r}\\x \text{ odd}}} \left\{\frac{x^2}{4p^r}\right\} = &\frac{1}{2}\left(p^r-p^{\left\lfloor\frac{r}{2}\right\rfloor}\right) + \left\{\frac{p^r}{4}\right\}p^{\left\lfloor\frac{r}{2}\right\rfloor} \nonumber \\
&-\frac{p^{\left\lfloor\frac{r+1}{2}\right\rfloor}-1}{2(p-1)}\left(\frac{1}{2}\left(1-(-1)^{\frac{p-1}{2}}\right)+\left(1-\left(\frac{2}{p}\right)\right)h_p^\ast+h_p^\prime\right).
\label{eq:sum2}
\end{align}
\end{prop}
\begin{proof}
The proof is similar to that of Proposition \cref{prop:sum1}.  For now, we leave out the terms where $p^r \mid x^2$.  We use the multiplicativity of $s_n(a)$ in $n$, which follows from the Chinese remainder theorem.  
\begin{align}
\sum_{\substack{0<x<2p^{r}\\x \text{ odd}\\p^r \nmid x^2}} \left\{\frac{x^2}{4p^r}\right\} 
&= \frac{1}{2} \sum_{\substack{0<y<4p^r\\2\nmid y\\p^r\nmid y}} s_{4p^r}(y) \frac{y}{4p^r} \nonumber \\
&= \frac{1}{2} \sum_{\substack{0<y<4p^r\\2\nmid y\\p^r\nmid y}} s_4(y)s_{p^r}(y) \frac{y}{4p^r} \nonumber\\
&= \frac{1}{2} \sum_{t=0}^{\lfloor\frac{r-1}{2}\rfloor} \sum_{\substack{0<a<4p^{r-2t}\\p\nmid a}} \left(1+(-1)^\frac{a-1}{2}\right)p^t\left(1+\left(\frac{a}{p}\right)\right) \frac{a}{4p^{r-2t}} \nonumber\\
&= \frac{1}{2} \sum_{t=0}^{\lfloor\frac{r-1}{2}\rfloor} p^t 
\left(\sum_{\substack{0<a<4p^{r-2t}\\2p\nmid a}} \frac{a}{4p^{r-2t}}
+\sum_{\substack{0<a<4p^{r-2t}\\2p\nmid a}} (-1)^{\frac{a-1}{2}} \frac{a}{4p^{r-2t}}\right.\nonumber\\
&\quad \left.+\sum_{\substack{0<a<4p^{r-2t}\\2p\nmid a}} \left(\frac{a}{p}\right) \frac{a}{4p^{r-2t}}
+\sum_{\substack{0<a<4p^{r-2t}\\2p\nmid a}} (-1)^{\frac{a-1}{2}}\left(\frac{a}{p}\right) \frac{a}{4p^{r-2t}}
\right).\label{eq:longsum}
\end{align}
We evaluate the first internal sum:
\begin{align*}
\sum_{\substack{0<a<4p^{r-2t}\\a \text{ odd}}} \frac{a}{4p^{r-2t}}
&= 
\sum_{a=0}^{4p^{r-2t}-1} \frac{a}{4p^{r-2t}}
-\sum_{a=0}^{2p^{r-2t}-1} \frac{a}{2p^{r-2t}}
-\sum_{a=0}^{4p^{r-2t-1}-1} \frac{a}{4p^{r-2t-1}}
+\sum_{a=0}^{2p^{r-2t-1}-1} \frac{a}{2p^{r-2t-1}} \\
&= \frac{4p^{r-2t}-1}{2} - \frac{2p^{r-2t}-1}{2} - \frac{4p^{r-2t-1}-1}{2} + \frac{2p^{r-2t-1}-1}{2} \\
&= p^{r-2t}-p^{r-2t-1}.
\end{align*}

The other internal sums may be evaluated using Lemma \cref{lem:imprim2}:
\begin{align*}
\sum_{\substack{0<a<4p^{r-2t}\\2p\nmid a}} (-1)^{\frac{a-1}{2}} \frac{a}{4p^{r-2t}}
&= \left(1-(-1)^{\frac{p-1}{2}}\right)\sum_{\substack{y=0\\2\nmid y}}^{3}  (-1)^{\frac{y-1}{2}}\frac{y}{4} \\
&= -\frac{1}{2}\left(1-(-1)^{\frac{p-1}{2}}\right);\\
\sum_{\substack{0<a<4p^{r-2t}\\2p\nmid a}} \left(\frac{a}{p}\right) \frac{a}{4p^{r-2t}}
&= \left(1-\left(\frac{2}{p}\right)\right)\sum_{\substack{y=0\\p\nmid y}}^{p-1} \left(\frac{y}{p}\right) \frac{y}{p} \\
&= -\left(1-\left(\frac{2}{p}\right)\right)h_p^\ast; \\
\sum_{\substack{0<a<4p^{r-2t}\\2p\nmid a}} (-1)^{\frac{a-1}{2}}\left(\frac{a}{p}\right) \frac{a}{4p^{r-2t}}
&= \sum_{\substack{a=0\\2p\nmid a}}^{4p} (-1)^{\frac{a-1}{2}}\left(\frac{a}{p}\right) \frac{a}{4p} \\
&= -h_p^\prime.
\end{align*}
After plugging the values of the internal sums into (\cref{eq:longsum}), we get a telescoping sum and a geometric series, yielding
\begin{align*}
\sum_{\substack{0<x<2p^{r}\\x \text{ odd}\\p^r \nmid x^2}} \left\{\frac{x^2}{4p^r}\right\} = \frac{1}{2}\left(p^r-p^{\left\lfloor\frac{r}{2}\right\rfloor}\right) 
-\frac{p^{\left\lfloor\frac{r+1}{2}\right\rfloor}-1}{2(p-1)}\left(\frac{1}{2}\left(1-(-1)^{\frac{p-1}{2}}\right)+\left(1-\left(\frac{2}{p}\right)\right)h_p^\ast+h_p^\prime\right). 
\end{align*}
Moreover, the terms we left out are those of the form $\left\{\frac{x^2/p^r}{4}\right\}$, where $x$ is odd and $x^2/p^r$ is an integer.  In this case, $x^2/p^r \equiv p^r \mod{4}$, so each of these terms is equal to $\left\{\frac{p^r}{4}\right\}$.  There are $p^{\left\lfloor \frac{r}{2} \right\rfloor}$ such terms, so
\begin{equation*}
\sum_{\substack{0<x<2p^{r}\\x \text{ odd}\\p^r \mid x^2}} \left\{\frac{x^2}{4p^r}\right\}
= \left\{\frac{p^r}{4}\right\}p^{\left\lfloor \frac{r}{2} \right\rfloor}.
\end{equation*}
The proposition follows. 
\end{proof}

\begin{thm}
\label{t:TrL}
Suppose $p$ is an odd prime, let $r \geq 1$, and let $L$ denote a standard choice of exponents for the cyclic Weil representation associated to the quadratic form $q(x) = \frac{1}{2m}x^2$, where $m=2p^r$.  Then,
\begin{equation}
\Tr(L) = p^r - \left\{\frac{p^r}{4}\right\}p^{\left\lfloor\frac{r}{2}\right\rfloor}
+\frac{p^{\left\lfloor \frac{r+1}{2}\right\rfloor}-1}{2(p-1)}\cdot\left\{\begin{array}{cc}
h_p & \mbox{ if } p \equiv 1 \mod 4 \\
4h_p^\ast + 1 & \mbox{ if } p \equiv 3 \mod 8 \\
2h_p + 1 & \mbox{ if } p \equiv 7 \mod 8 \\
\end{array}
\right..
\end{equation}
Here, $h_p^\ast = h_p$ except when $p=3$, in which case $h_p^\ast = \frac{1}{3}$.
\end{thm}
\begin{proof}
We have
\begin{align*}
\Tr(L) &= \sum_{x=1}^{2p^r} \left\{-\frac{x^2}{4p^r}\right\} \\
          &= \sum_{x=0}^{2p^r-1} \left\{-\frac{x^2}{4p^r}\right\} \\
          &= \sum_{\substack{0 \leq x < 2p^r\\2p^{\left\lfloor\frac{r+1}{2}\right\rfloor}\nmid x}} \left(1-\left\{\frac{x^2}{4p^r}\right\}\right) \\
          &= 2p^r - p^{\left\lfloor \frac{r}{2}\right\rfloor} - \sum_{x=0}^{2p^r-1} \left\{\frac{x^2}{4p^r}\right\}.
\end{align*}
Moreover,
\begin{align}
\sum_{x=0}^{2p^r-1} \left\{\frac{x^2}{4p^r}\right\}
&= \sum_{y=0}^{p^r-1} \left\{\frac{y^2}{p^r}\right\} + \sum_{\substack{0<x<2p^r\\x\text{ odd}}} \left\{\frac{x^2}{4p^r}\right\}.
\end{align}
These two sums have already been evaluated in Propositions \cref{prop:sum1} and \cref{prop:sum2}.  Plugging everything in, we get
\begin{equation*}
\Tr(L) = p^r - \left\{\frac{p^r}{4}\right\}p^{\left\lfloor\frac{p}{2}\right\rfloor} + \frac{p^{\left\lfloor\frac{r+1}{2}\right\rfloor}-1}{2(p-1)}\left(\frac{1}{2}\left(1-(-1)^{\frac{p-1}{2}}\right)+\left(3-\left(\frac{2}{p}\right)\right)h_p^\ast+h_p^\prime\right).
\end{equation*}
Splitting into cases modulo $8$, the theorem is proved.
\end{proof}

\begin{rmk}
It is well-known that there are an equal number of quadratic and nonquadratic residues between $0$ and $p/2$ when $p\equiv 1 \pmod{4}$, and that there are $\frac{p-1}{4} + h_p$ quadratic residues in that interval when $p \equiv 3 \pmod{4}$. Less well-known (see \cite{BerndtChowla}) is the fact that if $p \equiv 3\pmod{8}$, then there are an equal number of quadratic and nonquadratic residues in the interval between $0$ and $p/4$, and that if $p\equiv 7\pmod{8}$, then there are an equal number of quadratic and nonquadratic residues in the interval between $p/4$ and $p/2$. Our original computation of $\Tr(L)$ in the case $m = 2p$ made more explicit use of these facts to deduce our formulae for $\Tr(L)$ in those cases.

When $p\equiv 1 \pmod{4}$ we used the fact that
\[
h_p = -\sum_{a=0}^{4p-1}(-1)^{\frac{a-1}{2}}\left(\frac ap\right)\frac{a}{4p}.
\]
It is not hard to use this to prove that
\[
  h_p = 2\sum_{a=1}^{\frac{p-1}{4}}\left(\frac{a}{p}\right).
\]
Hence there are $\frac{p-1}{8} + \frac{1}{2}h_p$ quadratic residues between $0$ and $p/4$ when $p \equiv 1 \pmod{4}$. Thankfully, these various biases among the distributions of quadratic residues do not affect the main term of $\Tr(L)$, which is $p^r$ if $m = 2p^r$. This regularity in the main term will allow us to compute the asymptotic behaviour of generating weights.
\end{rmk}

\subsection{Critical weights}

Since the signature of $D=A_{2p^r}$ is odd, the only critical weights are $k=1/2,3/2$, which can be tackled using Theorem \ref{theorem:CriticalWeights}. The level of $D=A_{2p^r}$ is $N=4p^r$; thus, by Theorem \ref{theorem:CriticalWeights}, we have 
$$
\dim M_{1/2}(\rho_D) = \dim\left(\mathrm{Inv}(\rho_{A^{-1}_{2p^{r-1}}\oplus A^{-1}_{2p^r}})^{\epsilon\times 1} \oplus \mathrm{Inv}(\rho_{A^{-1}_{2p^r}\oplus A^{-1}_{2p^r}})^{\epsilon\times 1}\right).
$$
By Theorem 2 of \cite{SkoruppaCritical}, a necessary condition for $\mathrm{Inv}(\rho_{A^{-1}_{2l}\oplus A^{-1}_{2p^r}})\neq 0$ is that the size of $A^{-1}_{2l}\oplus A^{-1}_{2p^r}$ is a perfect square, i.e. $l=p^r$, therefore $\mathrm{Inv}(\rho_{A^{-1}_{2p^{r-1}}\oplus A^{-1}_{2p^r}}) = 0$. Another necessary condition in the same theorem  is that $\Omega_{A^{-1}_{2l}\oplus A^{-1}_{2p^r}}(1) = 1$. However, 
$$
\Omega_{A^{-1}_{2p^r}\oplus A^{-1}_{2p^r}}(1) =\Omega_{A^{-1}_{2p^r}}(1)^2 =  -i \neq 1,
$$
and thus 
\begin{equation}
\label{eqn:weight1/2}
M_{1/2}(\rho_{D}) = 0
\end{equation}
for all $p$ and $r$. To compute $\dim M_{3/2}(\rho_D)$, it follows from Theorem \ref{theorem:CriticalWeights} and the above discussion that 
$$
\dim S_{1/2}(\rho^*_{D}) =\dim \mathrm{Inv}(\rho_{A^{-1}_{2p^{r}}\oplus A_{2p^r}})^{\epsilon\times 1}  -\dim\mathrm{Inv}(\rho_{A^{-1}_{2p^r}\oplus A_{2p^r}})^{\mathrm{O}(A^{-1}_{2p^r})\times 1}.
$$
The terms of this difference might not necessarily vanish. However, we have:

\begin{prop}
Let $p$ be an odd prime, $r\in \ZZ_{>0}$. Then
$$
\mathrm{O}(A_{2p^r}^{-1}) = \{\pm 1\}.
$$
\end{prop}

\begin{proof}
By definition, 
$$
A^{-1}_{2p^r} = \left(\ZZ/2p^r\ZZ, x \mapsto \frac{-x^2}{4p^r}\right),
$$
and thus there is a bijection
$$
\mathrm{O}(A_{2p^r}^{-1}) \longleftrightarrow \{ x\in \{0,\ldots, 2p^r-1\} : x^2 = 1 \mod 4p^r\}.
$$
Now if $x^2 = 1 \mod 4p^r$ then $x = \pm 1 \mod p^r$, so that $x \in \{1,p^r-1,p^r+1,2p^r-1\}$. But we must also have $2\mid x + 1$ , which leaves only $x = 1,2p^r-1$. Both these numbers satisfy the congruence, thus $\mathrm{O}(A_{2p^r}^{-1}) \leftrightarrow \{ \pm 1\}$, as required. 
\end{proof}

It follows that $S_{1/2}(\rho^*_{A_{2p^r}}) = 0$ for all $p$ and $r$, since $-1\in \mathrm{O}(D)$ acts by $\epsilon$ on $\CC[D]$. Therefore 
\begin{equation}
\label{eqn:weight3/2}
\dim M_{3/2}(\rho_D) = \chi(\cMbar_{1/2}, \cW_{3/2}(D)).
\end{equation}

\subsection{Generating weights and distribution as $p\to\infty$}

Recall from Theorem \ref{thm:FreeModuleThm}, (iii) that since $\rho = \rho_{A_{2p^r}}$ is unitary, its generating weights all lie between $0,\ldots, 23/2$. We now give a simple formula for the multiplicity $m_j$ of the weight $j$ for each $j = 0,\ldots, 23/2$. To do so, consider the Poincar\'{e} generating series $\sum_{k\in \frac{1}{2}\ZZ}\dim M_k(\rho) t^k$, which can be expressed (by Theorem \ref{thm:FreeModuleThm}, (i)) as
\begin{equation}
\label{eqn:poincarePoly}
\sum_{k\in \frac{1}{2}\ZZ}\dim M_k(\rho) t^k = \frac{t^{k_1}+ \ldots + t^{k_{2p}}}{(1 - t^4)(1-t^6)}.
\end{equation}
Now we have:
\begin{itemize}
\item[(i)]If $k\in \ZZ$, $M_k(\rho) = 0$ since $\mathrm{sig}(A_{2p^r})$ is odd (by \eqref{eqn:parityArgument}).
\item[(ii)] If $k<0$, $M_k(\rho) = 0$ since $\rho$ is unitary (Theorem \ref{thm:FreeModuleThm}, (iii)).
\item[(iii)] $M_{1/2}(\rho) = 0$ by \eqref{eqn:weight1/2}.
\item[(iv)] If $k\in 1/2 + \ZZ$ and $k\geq 3/2$, $M_k(\rho) = \chi(\cMbar_{1/2},\cW(\cA_{2p^r}))$, by \eqref{eqn:weight3/2}  and Serre duality. 
\end{itemize}
We therefore know how to compute $M_k(\rho)$ for all $k\in \frac{1}{2}\ZZ$, and we can then solve for the generating weights $k_j$ and their multiplicities $m_j$ by comparing coefficients in \eqref{eqn:poincarePoly}. In particular, note that we may immediately deduce from (i) above that $m_j = 0$ if $j\in \ZZ$, $j<0$ or $j=1/2$. The complete list of multiplicities is given by Table \ref{table:weightMultiplicities}. Recall that we have assumed $p\geq 5$ (though similar formulas can easily be obtained for $p=2,3$) and in the table we use the notation 
$$
\delta \df \frac{1}{8}\left( 2 + \left( \frac{-1}{p^r} \right)\right), \quad 
\epsilon_{\pm} \df \frac{1}{6}\left( 1 \pm \left( \frac{p^r}{3} \right)\right).
$$

\renewcommand{\arraystretch}{1.5}
\begin{center}
\begin{table}[h]
\begin{tabular}{r|l}
\textrm{Weights $k$ } & \textrm{Multiplicities $m_k$}\\
\hline
$1/2$ & $0$ \\
$3/2$& $\frac{13}{24}(p^r+1) - \frac{1}{2}\Tr(L) - \delta  - \epsilon_+$  \\
$5/2$& $\frac{15}{24}(p^r-1) - \frac{1}{2}\Tr(L) + \delta$   \\
$7/2$ & $\frac{17}{24}(p^r+1) - \frac{1}{2}\Tr(L) - \delta  + \epsilon_+$  \\
$9/2$&  $\frac{19}{24}(p^r-1) - \frac{1}{2}\Tr(L) + \delta  + \epsilon_-  $  \\
$11/2$&  $\frac{1}{3}(p^r+1) + \epsilon_+ $   \\
$13/2$&   $\frac{1}{3}(p^r-1) - \epsilon_- $ \\
$15/2$& $-\frac{5}{24}(p^r+1) + \frac{1}{2}\Tr(L) + \delta - \epsilon_+ $\\
$17/2$& $-\frac{7}{24}(p^r-1) + \frac{1}{2}\Tr(L) - \delta - \epsilon_- $\\
$19/2$ & $-\frac{9}{24}(p^r+1) + \frac{1}{2}\Tr(L) + \delta $ \\
$21/2$& $-\frac{11}{24}(p^r-1) + \frac{1}{2}\Tr(L) - \delta + \epsilon_-$\\
$23/2$& $0$ \\
\end{tabular}
\vspace{12pt}
\caption{Weight multiplicities for $\rho = \rho_{A_{2p^r}}$, $p\geq 5$.}
\label{table:weightMultiplicities}
\end{table}
\end{center}

A lot of information can be extracted from Table \ref{table:weightMultiplicities}. First, we record an immediate corollary:
\begin{cor}
The weight $k = 23/2$ never occurs as the generating weight of the free module of vector-valued modular modular forms for $\rho = \rho_{A_{2p^r}}$ (equiv. Jacobi forms of index $p^r$). 
\end{cor} 

While the fact that the weight $k=1/2$ does not occur as a generating weight is well-known (\cite{EichlerZagier} Theorem 5.7, as it follows from Skoruppa's Theorem \ref{theorem:CriticalWeights} above) the fact that the same is true for the weight $k=23/2$ does not seem to appear anywhere else in the literature concerning vector-valued modular forms for the Weil representation or Jacobi forms.

This `endpoint' symmetry in the number of generating weights suggest that there is a deeper, hidden symmetry in the way the generating weights distribute as $p\to\infty$. In fact, this distribution can be computed explicitly from Table \ref{table:weightMultiplicities}. To elucidate this point, fix $r$ and let $p\to\infty$. As illustrated in Figure \ref{figure:weightDitribution} (generated using {\em Wolfram Mathematica} v. 10.2) for the case $r=1$, it looks like the generating weights have a bi-modal distribution centered around $k=11/2,13/2$. 
\begin{figure}[h]
\centering
\includegraphics[scale=0.75]{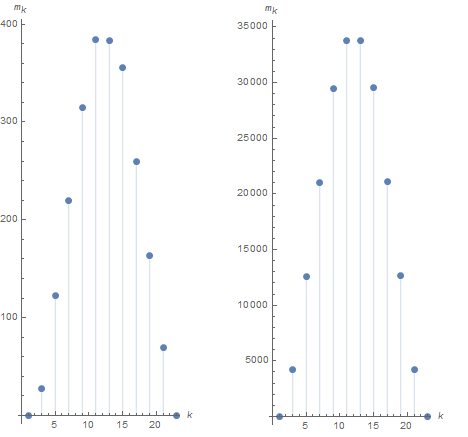}
\caption{Distribution of generating weights for $p=1151, 101281$ and $r=1$. Note that $h_{1151} = 41$ and $h_{101281} = 168$.}
\label{figure:weightDitribution}
\end{figure}

In fact, using our exact formula for $\Tr(L)$ and well-known upper bounds on class numbers of imaginary quadratic fields, we are able to prove the following theorem concerning the precise distribution of the generating weights as $p\to\infty$:

\begin{cor}
\label{c:limitingroots}
Let $\rho = \rho_{A_{2p^r}}$, where $p>3$ is a prime, as above. Let $m_k$ denote the multiplicity of the generating weight $k \in \frac 12 \ZZ$ for $M(\rho)$. Then the values of $\lim_{p\to \infty} \frac{m_k}{\dim \rho}$ are as follows:

\begin{center}
\rm{\begin{tabular}{l | c c c c c c c c c c c c}
Weight $k$ & $\sfrac{1}{2}$ & $\sfrac{3}{2}$ & $\sfrac{5}{2}$ & $\sfrac{7}{2}$ & $\sfrac{9}{2}$ & $\sfrac{11}{2}$ & $\sfrac{13}{2}$ & $\sfrac{15}{2}$ & $\sfrac{17}{2}$ & $\sfrac{19}{2}$ & $\sfrac{21}{2}$ & $\sfrac{23}{2}$ \\
\hline
$\displaystyle{\lim_{p\to \infty}} \frac{m_k}{\dim \rho}$ & $0$ & $\sfrac{1}{48}$ & $\sfrac{3}{48}$ & $\sfrac{5}{48}$ & $\sfrac{7}{48}$ & $\sfrac{8}{48}$ & $\sfrac{8}{48}$ & $\sfrac{7}{48}$ & $\sfrac{5}{48}$ & $\sfrac{3}{48}$ & $\sfrac{1}{48}$ & $0$
\end{tabular}}
\end{center}

\end{cor}

\begin{proof}
Immediate from Table \ref{table:weightMultiplicities} upon noting that $\lim_{p \to \infty} h_p/p^{\frac 12+\veps} = 0$ for all $\veps > 0$.
\end{proof}

\begin{rmk}
While preparing \cite{FrancMason:Structure}, the authors performed similar computations for representations of $\SL_2(\ZZ)$ arising from irreducible representations of $\SL_2(\FF_p)$. If one restricts to representations of $\PSL_2(\ZZ)$, so that weight one is not an issue, then the weight multiplicity formulae of \cite{CandeloriFranc} allow one to prove that whenever one has a family of representations $\rho_n$ such that the trivial representation does not occur in the $\rho_n$, and the main term of the trace of the exponents approaches $\frac 12 \dim \rho_n$, then the generating weights have the limiting distribution $(0,\frac 1{12}, \frac{3}{12}, \frac{4}{12}, \frac{3}{12},\frac {1}{12},0)$ for weights $(0,2,4,6,8,10,12)$, respectively. It would be of interest to find a more conceptual explanation for these distributions. 
\end{rmk}

\renewcommand\refname{References}

\bibliographystyle{alpha}
\bibliography{RootsWeil}

\end{document}